\documentclass[12pt]{article}
\usepackage{enumerate,amsthm,color}
\usepackage{mathrsfs}
\usepackage{epic,eepic,epsf,epsfig}
\usepackage{multirow}
\usepackage{amsmath}
\usepackage{amssymb}
\usepackage{amsbsy}
\usepackage{graphicx}
\usepackage{amsfonts}

\newtheorem{lem}{Lemma}[section]
\newtheorem{thm}[lem]{Theorem}

\newtheorem{prob}[lem]{Problem}
\newtheorem{prop}[lem]{Proposition}

\def\a{\alpha}\def\b{\beta}\def\d{\delta}\def\s{\sigma}
\def\BC{{\cal BC}}

\def\G{\Gamma}

\def\Z{\mathbb{Z}}

\def\Aut{\hbox{\rm Aut}}
\def\BiCay{\hbox{\rm BiCay}}
\def\Cay{\hbox{\rm Cay}}
\def\H{\hbox{\rm H}}

\newcommand{\sg}[1]{\langle #1 \rangle}

\begin{document}
\title{
 Existence of non-Cayley Haar graphs}
\footnotetext{The first author was partially supported by the National Natural Science Foundation of China (11731002, 11571035) and by the 111 Project of China (B16002). The second author was partially supported by the Slovenian Research Agency (research program P1-0285, and research projects N1-0062, J1-9108 and J1-1695). The third author was partially supported by the National Natural Science Foundation of China (11731002). The fourth author was partially supported by the Fundamental Research Funds for the Central Universities and Innovation Foundation of BUPT for Youth (500419775).

*Corresponding author. E-mails: yqfeng@bjtu.edu.cn (Y.-Q.\ Feng), istvan.kovacs@upr.si (I.\ Kov\'acs),
wangj@pku.edu.cn (J.\ Wang), dwyang@bjtu.edu.cn
(D.-W.\ Yang).}

\author{Yan-Quan~Feng$^{\, a},$ Istv\'an~Kov\'acs$^{\, b},$ Jie Wang$^{\, c}$,
Da-Wei Yang*$^{\, d}$ \\ [+1ex]
{\small\em $^a$ Department of Mathematics, Beijing Jiaotong University, Beijing, 100044, P.~R.~China} \\
{\small\em $^b$ IAM and FAMNIT, University of Primorska,
Glakolja\v{s}ka 8, 6000 Koper, Slovenia}\\
{\small\em $^c$ School of Mathematical Sciences, Peking University, Beijing, 100871, P.~R.~China}\\
{\small\em $^d$ School of Sciences, Beijing University of Posts and Telecommunications, Beijing, 100876, P.R. China}
}
\date{}
\maketitle

\begin{abstract}
A Cayley graph of a group $H$ is a finite simple graph $\G$ such
that its automorphism group $\Aut(\G)$ contains a subgroup isomorphic to $H$ acting regularly on $V(\G)$, while a Haar graph of $H$ is
a finite simple bipartite graph $\Sigma$ such that $\Aut(\Sigma)$ contains a subgroup isomorphic to $H$ acting semiregularly on
$V(\Sigma)$ and the $H$-orbits are equal to the partite sets of
$\Sigma$. It is well-known that every Haar graph of finite
abelian groups is a Cayley graph. In this paper, we
prove that every finite non-abelian group admits a non-Cayley Haar graph except the dihedral groups $D_6$, $D_8$, $D_{10}$,
the quaternion group $Q_8$ and the group $Q_8\times\Z_2$. This answers
an open problem proposed by Est\'elyi and Pisanski in 2016.
\\ [+1.5ex]
\\ [+1.5ex]
\end{abstract}

\section{Introduction}

All groups in this paper are finite and all graphs
are finite and undirected. Let $H$ be a group, and let $R$, $L$ and $S$ be three subsets of $H$ such that $R^{-1}=R$, $L^{-1}=L$, and $R\cup L$
does not contain the identity element $1$ of $H$.
The {\em Cayley graph} of $H$ relative to the subset $R,$ denoted by $\Cay(H,R),$ is the graph having vertex set
$H,$ and edge set $\{ \{h,xh\} : x \in R, h \in H \}$, and the {\em bi-Cayley graph} of $H$ relative to the triple $(R, L, S)$, denoted by $\BiCay(H, R, L, S)$, is the graph having vertex set the union of the right part
$H_0=\{h_0 : h\in H\}$ and the left part $H_1=\{h_1 : h\in H\}$, and edge set being the union of the following three sets
\begin{itemize}
\item $\big\{ \{h_0,(xh)_0\} : x \in R, \, h \in H \big\}$ (right edges),
\item $\big\{ \{h_1,(xh)_1\}:  x \in L, \, h \in H \big\}$ (left edges),
\item $\big\{ \{h_0,(xh)_1\} : x \in S, \, h \in H\big\}$ (spokes).
\end{itemize}
In the special case when $R=L=\emptyset$, the bi-Cayley graph
$\BiCay(H, \emptyset, \emptyset, S)$ is called a {\em Haar graph of $H$ relative to the set $S$}, denoted by $\H(H,S)$.
A Haar graph $H(H,S)$ of a finite group $H$ was first defined as a voltage graph of a dipole
with no loops and $|S|$ parallel edges (see~\cite{HMP}), and the name {\em Haar graph}
comes from the fact that, when $H$ is an abelian group the Schur norm of the corresponding adjacency matrix can be easily evaluated via the so called Haar integral on $H$ (see~\cite{H}).

Symmetries of Cayley graphs have always been an active topic among algebraic combinatorics, and lately, the symmetries of bi-Cayley graphs received considerable attention. For various results and constructions in connection with bi-Cayley graphs and their automorphisms, we refer the reader to \cite{AHK,AKKMM,CEP,EP,KK,LWX1,ZF,ZF1} and all the references therein. In particular, Est\'elyi and Pisanski~\cite{EP} initiated the investigation for the relationship between Cayley graphs and Haar graphs. A Cayley graph is a Haar graph exactly when it is bipartite, but no simple condition is known for a Haar graph to be a Cayley graph.
An elementary argument shows that every Haar graph of abelian groups is a Cayley graph (
this also follows from Proposition~\ref{ZF}).  On the other hand, Lu et al.~\cite{LWX} constructed cubic semi-symmetric graphs, that is, edge- but not vertex-transitive graphs, as Haar graphs of alternating groups.
Clearly, as these graphs are not vertex-transitive, they are examples of Haar graphs which are not Cayley graphs.
It is natural to ask which non-abelian groups admit a Haar graph that is not a Cayley graph, or putting it another way, we have the following problem, which was first posed by Est\'elyi and Pisanski~\cite[Problem~1]{EP}.

\begin{prob}\label{EP1}
{\rm (\cite{EP})} Determine the finite non-abelian groups $H$ for which all Haar graphs  $\H(H,S)$ are Cayley graphs.
\end{prob}

We denote by $\Z_n$ the cyclic group of order $n$, by $D_{2n}$ the dihedral group of order $2n,$ and by $Q_8$ the quaternion group.
Est\'elyi and Pisanski~\cite[Theorem~8]{EP} solved Problem~\ref{EP1} for dihedral
groups.

\begin{prop}{\rm (\cite{EP})}\label{Prop=1}
Each Haar graph of the dihedral group $D_{2n}$ is a Cayley graph if and only if $n=2,3,4,5$.
\end{prop}

A group $H$ is called {\em inner abelian} if $H$ is non-abelian, and all proper subgroups of $H$ are abelian. Recently, Feng et al.~\cite[Theorem~1.2]{FKY} solved Problem~\ref{EP1} for the class of inner abelian groups. 

\begin{prop}{\rm (\cite{FKY})}\label{Prop=Inner abelian group}
Each Haar graph of an inner abelian group $H$ is a Cayley graph if and only if $H\cong D_6$,
$D_8$, $D_{10}$ or $Q_8$.
\end{prop}

In this paper we solve Problem~\ref{EP1} completely.

\begin{thm}\label{1}
Let $H$ be a non-abelian group with the property that every Haar graph of $H$ is a Cayley graph. Then $H$ is isomorphic to $D_{6}$, $D_8$, $D_{10}$, $Q_8$ or $Q_8\times\Z_2$.
\end{thm}

The main idea of the proof of Theorem~\ref{1} is to construct non-Cayley Haar graphs.
It is worth mentioning that all non-Cayley Haar graphs of non-abelian groups, constructed in~\cite{EP,FKY} and this paper,
are not vertex-transitive. It seems difficulty to construct vertex-transitive non-Cayley Haar graphs. Est\'elyi and Pisanski~\cite{EP} raised a question
whether there exists a vertex-transitive non-Cayley Haar graph. Later, infinitely many vertex-transitive non-Cayley Haar graphs were constructed by Conder et al.~\cite{CEP}
and Feng et al.~\cite{FKMY}, and this prompts us to consider the following problem.

\begin{prob}\label{P}
Determine the finite non-abelian groups $H$ for which all vertex-transitive Haar graphs  $\H(H,S)$ are Cayley graphs.
\end{prob}

Note that Problem~\ref{P} is closely related to the so called non-Cayley numbers. A positive integer $n$ is called a {\em Cayley number} if every vertex-transitive graph of order $n$ is a Cayley graph, and otherwise it is a {\em non-Cayley number}. In 1983, Maru\v si\v c~\cite{M} posed the problem of determining Cayley numbers,
and this question has generated a fair amount of interests.
For some works about Cayley numbers and vertex-transitive non-Cayley graphs, one may refer to~\cite{DS,LS,ZF2}. 

By a {\em graphical regular representation} (GRR for short) for a group $H$ we mean a Cayley graph $\G$ of $H$ such that $\Aut(\G)\cong H$.
When studying a Cayley graph $\G$ of a finite group $H$, a very important question is to determine whether $H$ is in fact the full automorphism group of $\G$.
For this reason, GRRs have been widely studied. The most natural question is classifying finite groups admitting a GRR, and the solution was derived in several papers (see, for instance, \cite{B,Gd,I,I1,IW,NW,NW1,Wa}).
A bi-Cayley graph $\Sigma$ of a group $H$ is called a {\em bi-graphical regular representation} (bi-GRR for short) if $\Aut(\Sigma)\cong H$. The problem of classifying finite groups admitting a bi-GRR was posed by Zhou~\cite{Z} (also see~\cite{HKM}), and it was solved by Du et al.~\cite{DFP} recently.
Motivated by GRR and bi-GRR, a {\em GHRR} of a group $H$ is a Haar graph $\G$ of $H$ with $\Aut(\G)\cong H$. Since
every Haar graph of abelian groups is a Cayley graph, abelian groups have no GHRR.
However, many non-abelian groups have GHRRs,
for example, see \cite{EP,FKY} and Section~3 of this paper.
Moreover, Theorem~\ref{1} implies that the non-abelian groups $D_6$, $D_8$, $D_{10}$, $Q_8$ and $Q_8\times\Z_2$ have no GHRRs, and
to the best of our knowledge, they are the only known non-abelian groups that have no GHRRs. In the end of this section, we would like to pose the following problem.

\begin{prob}
Determine the
finite non-abelian groups that have no GHRRs.
\end{prob}

The rest of the paper is organized as follows. In the next section we collect all concepts and results that will be used later. In Section 3, we introduce some Haar graphs that are not vertex-transitive, and prove Theorem~\ref{1} in Section~4.

\section{Preliminaries}

 For a graph $\G$, we denote by
$V(\G),$ $E(\G)$ and $\Aut(\G)$ the vertex set, the edge set and the group of all automorphisms of $\G$. Given a vertex $v \in V(\G),$ we
denote by $\G(v)$ the set of vertices adjacent to $v$.
For a subgroup $G$ of $\Aut(\G),$ denote by $G_v$ the stabilizer of the vertex $v$ in $G$, that is, the subgroup of $G$ fixing $v$.
We say that $G$ is {\em semiregular} on
$V(\G)$ if $G_{v}=1$ for every $v\in V(\G)$, and {\em regular} if $G$ is transitive and semiregular.

Let $\G=\H(H,S)$ be a Haar graph of a group $H$ with
identity element $1$. By \cite[Lemma~3.1(2)]{ZF}, up to graph isomorphism, we may always assume that $1\in S$. The graph
$\G$ is then connected exactly when $H=\sg{S}$.
For $g \in H$, the {\em right translation} $R(g)$ is the permutation
of $H$ defined by $R(g) : h \mapsto hg$ for $h \in H,$ and the
{\em left translation} $L(g)$ is
the permutation of $H$ defined by
$L(g) : h \mapsto g^{-1}h$ for $h \in H$.
Set $R(H)=\{ R(h) \, : \, h \in H\}$. Recall that $V(\G)=H_0\cup H_1$. It is easy to see that $R(H)$ can be regarded as a group of automorphisms of $\H(H,S)$ acting on $V(\G)$ by the rule
$$
R(g):\ h_i\mapsto (hg)_i,\ \forall i\in \{0,1\},~\forall h,g\in H.
$$
Furthermore, $R(H)$ acts semiregularly on $V(\G)$ with two orbits $H_0$ and $H_1$.

For an automorphism $\a \in \Aut(H)$ and $x, y, g \in H$,
define two permutations on $V(\G)=H_0 \cup H_1$ as follows
\begin{eqnarray}
& \d_{\a,x,y}:\ h_0\mapsto (xh^\a)_1,~h_1\mapsto (yh^\a)_0\label{Eq-delta},& \forall h\in H;  \\
&  \s_{\a,g}:\ h_0\mapsto (h^\a)_0,~h_1\mapsto (gh^\a)_1, & \forall h\in H. \label{Eq-sigma}
\end{eqnarray}
Set
\begin{eqnarray*}
{\rm I} &=& \{\d_{\a,x,y} : \a \in\Aut(H),~S^\a=y^{-1}S^{-1}x\},
\label{Eq-d}
\\
{\rm F} &=& \{\s_{\a,g} : \a \in \Aut(H),~S^\a=g^{-1}S \}.
\label{Eq-s}
\end{eqnarray*}
By~\cite[Lemma~3.3]{ZF}, ${\rm F}\leq \Aut(\G)_{1_0}$, and if $\G$ is connected, then ${\rm F}$ acts on the set $\G(1_0)$ consisting of all neighbours of $1_0$ faithfully. By~\cite[Theorem~1.1 and Lemma~3.2]{ZF}, we have the following proposition.

\begin{prop}\label{ZF}
Let $\G=\H(H,S)$ be a connected Haar graph, and let $A=\Aut(\G)$.
\begin{enumerate}[\rm (i)]
\item If ${\rm I}=\emptyset,$ then the normalizer
$N_{A}(R(H))=R(H)\rtimes {\rm F}$.
\item If ${\rm I}\neq \emptyset$,
then $N_{A}(R(H))=R(H)\sg{{\rm F},\d_{\a,x,y}}$
for some $\d_{\alpha,x,y}\in {\rm I}$.
\end{enumerate}
Moreover, $\sg{R(H),\d_{\a,x,y}}$ acts transitively on $V(\G)$ for any
$\d_{\a,x,y}\in {\rm I}$.
\end{prop}

Throughout the paper we follow the notation defined in~\cite{FKY}:
$$
\BC=\big\{ H \; \text{is a finite group} : \H(H,S) \; \text{is a Cayley graph for any }S \subseteq H \big\}.
$$
The following proposition was given by \cite[Lemma~3.1]{FKY}.

\begin{prop}\label{L1}
The class $\BC$ is closed under taking subgroups.
\end{prop}

In view of~\cite[Theorem~1.3 and Corollary~4.6]{FKY}, we have the following proposition.

\begin{prop}\label{Prop=2}
Let $H$ be a group belonging to the class $\BC$. Then the following hold.
\begin{enumerate}[\rm (i)]
\item The group $H$ is solvable.
\item Each Sylow $p$-subgroup of $H$ with a prime $p\geq3$ is abelian.
\item If $H$ is non-abelian, then $H$ has a subgroup isomorphic to $D_6$,
$D_8$, $D_{10}$ or $Q_8$.
\end{enumerate}
\end{prop}

The following proposition is well-known, and one may see~\cite[(1.12)]{As}.

\begin{prop}\label{prop=abelian p-group}
Let $P$ be a finite abelian $p$-group.
Then $P=\Z_{p^{e_1}}\times \Z_{p^{e_2}}\times\cdots\times \Z_{p^{e_n}}$,
where $1\leq e_1\leq e_2\leq \cdots\leq e_n$. Moreover, the integers
$n$ and $e_i$ with $1\leq i\leq n$ are uniquely determined by $P$.
\end{prop}

\section{Haar graphs that are not vertex-transitive}

In this section, we introduce some Haar graphs that are not vertex-transitive, which
will be used in the proof of Theorem~\ref{1}. First we describe two infinite families of Haar graphs that are not vertex-transitive.

\begin{lem}\label{L-2}
Let $n$ be an integer with $n\geq3$,
and let $p$ be an odd prime. Let
$$H=D_{2n}\times\Z_p=\sg{a,b,c~|~a^n=b^2=c^p=[a,c]=[b,c]=1, a^b=a^{-1}},$$
and $S=\{1,a,b,c,abc\}$. Then $\Aut(\H(H,S))=R(H)$ and $H\not\in \BC$.
\end{lem}

\begin{proof} Let $\G=\H(H,S)$ and let $A=\Aut(\G)$.
Note that $R(H)\leq A$ has exactly two orbits on
$V(\G)$. Then $A$ is vertex-transitive or has two orbits, that is, $H_0$ and $H_1$. For the former, $A_{1_0}$ and $A_{1_1}$ are conjugate in $A$, and for the latter, the Frattini argument implies that $A=R(H)A_{1_0}=R(H)A_{1_1}$. In the both cases, $|A_{1_0}|=|A_{1_1}|$, and hence $|A_{1_0}|=|A_{h_0}|=|A_{k_1}|$ for any $h,k\in H$. To finish the proof, it suffices to show that $A_{1_0}=1$ and $\G$ is not vertex-transitive.

We depicted the subgraph of $\G$ induced  by the vertices
at distance at most $2$ from $1_0$ in Figure~\ref{Fig-1}.

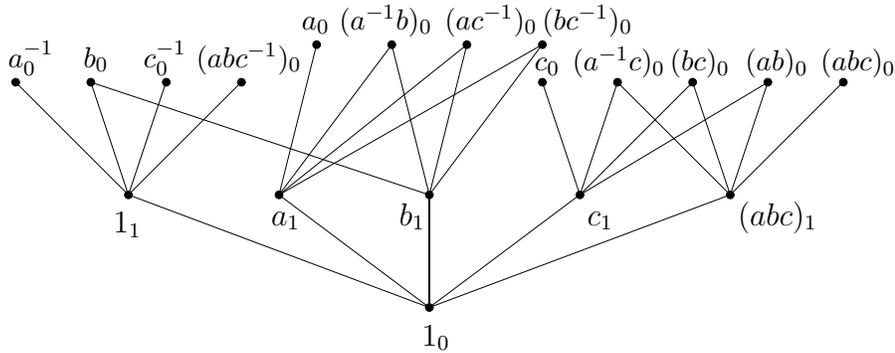
\begin{figure}[h]
\begin{center}
\unitlength 1mm
\begin{picture}(150,40)

\put(70, 0) {\circle*{1}} \put(69,-5){$1_0$}
\put(30, 15){\circle*{1}} \put(28,10){$1_1$}
\put(70,0){\line(-8, 3){40}}

\put(50, 15){\circle*{1}} \put(49,11){$a_1$}
\put(70,0){\line(-4, 3){20}}

\put(70, 15){\circle*{1}} \put(66,11){$b_1$}
\put(70,0){\line(0, 1){15}}

\put(90, 15){\circle*{1}} \put(91,11){$c_1$}
\put(70,0){\line(4, 3){20}}

\put(110, 15){\circle*{1}} \put(111,11){$(abc)_1$}
\put(70,0){\line(8, 3){40}}

\put(15,30){\circle*{1}} \put(14,32){{\small$a^{-1}_0$}}
\put(30,15){\line(-1, 1){15}}
\put(25,30){\circle*{1}} \put(24,32){{\small$b_0$}}
\put(30,15){\line(-1, 3){5}}
\put(35,30){\circle*{1}} \put(32,32){{\small$c^{-1}_0$}}
\put(30,15){\line(1, 3){5}}
\put(45,30){\circle*{1}} \put(39,32){{\small $(abc^{-1})_0$}}
\put(30,15){\line(1, 1){15}}

\put(55,35){\circle*{1}} \put(53,37){{\small$a_0$}}
\put(50,15){\line(1, 4){5}}
\put(65,35){\circle*{1}} \put(58,37){{\small $(a^{-1}b)_0$}}
\put(50,15){\line(3, 4){15}}
\put(75,35){\circle*{1}} \put(72,37){{\small $(ac^{-1})_0$}}
\put(50,15){\line(5, 4){25}}
\put(85,35){\circle*{1}} \put(85,37){{\small $(bc^{-1})_0$}}
\put(50,15){\line(7, 4){35}}

\put(70,15){\line(-3, 1){45}}  \put(70,15){\line(-1, 4){5}}
\put(70,15){\line(1, 4){5}} \put(70,15){\line(3, 4){15}}

\put(85,30){\circle*{1}} \put(84,31.5){{\small $c_0$}}
\put(90,15){\line(-1, 3){5}}
\put(95,30){\circle*{1}} \put(89,31.5){{\small $(a^{-1}c)_0$}}
\put(90,15){\line(1, 3){5}}
\put(105,30){\circle*{1}} \put(102,31.5){{\small $(bc)_0$}}
\put(90,15){\line(1, 1){15}}
\put(115,30){\circle*{1}} \put(112,31.5){{\small $(ab)_0$}}
\put(90,15){\line(5, 3){25}}

\put(125,30){\circle*{1}} \put(122,31.5){{\small $(abc)_0$}}
\put(110,15){\line(1, 1){15}} \put(110,15){\line(-1, 1){15}}
\put(110,15){\line(-1, 3){5}} \put(110,15){\line(1, 3){5}}

\end{picture}
\end{center}\vspace{-0cm}
\caption{The subgraph of $\G$ induced
by the vertices at distance at most $2$ from $1_0$.}\label{Fig-1}
\end{figure}

Consider the $4$-cycles of $\G$ passing through the vertex $1_0$. For each $h\in H$, denote by $\G(h_0)$ and $\G(h_1)$ the neighborhoods of $h_0$ and $h_1$ in $\G$ respectively, that is, $\G(h_0)=\{(sh)_1~|~s\in S\}$ and $\G(h_1)=\{(s^{-1}h)_0~|~s\in S\}$.
By Figure~\ref{Fig-1}, the numbers of $4$-cycles passing through the edges $\{1_0,1_1\}$ and $\{1_0,b_1\}$
are $1$ and $4$, respectively, while there are exactly three $4$-cycles
passing through the edge $\{1_0,u_1\}$ for each $u_1=a_1,c_1$ or $(abc)_1$. This implies that $A_{1_0}$ fixes $1_1$ and $b_1$, and $\{a_1,c_1,(abc)_1\}$ setwise. It follows that $A_{1_0}\leq A_{1_1}$
and $A_{1_0}\leq A_{b_1}$, and since there is a unique $4$-cycle passing through $1_0$ and $1_1$, we have $A_{1_0}\leq A_{b_0}$. Since $|A_{1_0}|=|A_{h_0}|=|A_{k_1}|$ for any $h,k\in H$, we have $A_{1_0}=A_{1_1}=A_{b_1}=A_{b_0}$.

By Figure~\ref{Fig-1}, there are $4$-cycles passing through $(a_1,1_0,b_1)$ but no $4$-cycles passing through
$(c_1,1_0,b_1)$ or $((abc)_1,1_0,b_1)$, and since $A_{1_0}$ fixes $b_1$ and $\{a_1,c_1,(abc)_1\}$ setwise, $A_{1_0}$ fixes $a_1$, and $\{c_1,(abc)_1\}$ setwise. Thus, $A_{1_0}$ fixes $\G(a_1)$ setwise, and since there exist $4$-cycles passing through $1_0$, $a_1$ and a vertex in $\G(a_1)$ except $a_0$, we have $A_{1_0}\leq A_{a_0}$.
It follows that $A_{1_0}=A_{a_0}=A_{a_1}$.

Now we claim that $A_{1_0}$ fixes  $c_1$ and $(abc)_1$. Note that $A_{1_0}$ fixes $\{c_1,(abc)_1\}$ setwise. Suppose that $\a\in A_{1_0}$ interchanges $c_1$ and $(abc)_1$. By Figure~\ref{Fig-1}, there exist $4$-cycles passing through $1_0$, $c_1$ (resp. $(abc)_1$) and a vertex in $\G(c_1)$ (resp. $\G((abc)_1)$) except $c_0$ (resp. $(abc)_0$), and hence  $\a$ interchanges $c_0$ and $(abc)_0$. Since $A_{1_0}$ fixes $a_0$, we have $(\G(a_0)\cap  \G(c_0))^\a=\G(a_0)\cap  \G((abc)_0)$. Clearly, $(ac)_1\in \G(a_0)\cap \G(c_0)$. Then $|\G(a_0)\cap  \G((abc)_0)|\not=0$, and hence there exist
$s,t\in S$ such that $sa=tabc$, that is, $t^{-1}s=a^2bc\in S^{-1}S$.
This is impossible as $S=\{1,a,b,c,abc\}$. Thus $A_{1_0}$ fixes $c_1$ and $(abc)_1$, and hence $c_0$ and $(abc)_0$. It follows that  $A_{1_0}=A_{c_0}=A_{c_1}$.

Now we have that $A_{1_0}=A_{x_0}$ for each $x\in T:=\{a,b,c\}$. For any $y\in T$, we have $A_{1_0}^{R(y)}=A_{x_0}^{R(y)}$, that is, $A_{y_0}=A_{(xy)_0}$. It follows that $A_{1_0}=A_{(xy)_0}$, and an easy  inductive argument implies that $A_{1_0}=A_{(x_1x_2\cdots x_n)_0}$ for any $x_1,\cdots,x_n\in T$. Since $\langle T\rangle=H$, $A_{1_0}$ fixes $H_0$ pointwise. Since $A_{1_0}=A_{1_1}$, we have $A_{h_0}=A_{h_1}$ for any $h\in H$, and hence $A_{1_0}$ fixes $H_1$ pointwise. Thus, $A_{1_0}=1$.

\medskip
To finish the proof, we are left with showing that $A$ is not vertex-transitive. Suppose to the contrary that $A$ is vertex-transitive. Since $A_{1_0}=1$, we have $|A|=|V(\G)|=2|R(H)|$ and hence
$R(H)\unlhd A$. By Proposition~\ref{ZF}, there exists $\d_{\b,x,y}\in A$
for some $\b\in \Aut(H)$ and $x,y\in H$ such that $S^{\b}=y^{-1}S^{-1}x$. Since
$R(H)$ acts transitively on $H_1,$ we may further assume that $1_0^{\d_{\b,x,y}}=1_1$. By Eq.~\eqref{Eq-delta}, $1_0^{\d_{\b,x,y}}=(x1^{\b})_1=1_1$, forcing $x=1$. Thus $S^{\b}=y^{-1}S^{-1}$, that is,
\begin{equation}\label{Eq-11}
S^\b=\{1^{\b},a^{\b},b^{\b},c^{\b},(abc)^{\b}\}=
y^{-1}\{1,a^{-1},b,c^{-1},abc^{-1}\}.
\end{equation}
Since $1\in S$, we have $1\in S^\b$ and so $y^{-1}=1,a,b^{-1}, c$ or $abc$.

Note that $H=D_{2n}\times\Z_p=\sg{a,b,c~|~a^n=b^2=c^p=[a,c]=[b,c]=1, a^b=a^{-1}}$. If $n$ is odd then the center $Z(H)=\Z_p$, and  if $n=2m$ is even then $Z(H)=\sg{a^m}\times\Z_p\cong\Z_2\times\Z_p$, where $\Z_p$ is characteristic in $\sg{a^m}\times\Z_p$. It follows that $\Z_p=\sg{c}$ is characteristic in $H$, and since $\b\in\Aut(H)$, we have $c^\b\in\sg{c}$.

If $y^{-1}=a,b^{-1}, c$ or $abc$, we have from Eq.~(\ref{Eq-11}) that $S^\b=\{a,1,ab,ac^{-1},a^2bc^{-1}\}$, $\{b, ba^{-1}$, $1$, $bc^{-1}$, $a^{-1}c^{-1}\}, \{c,ca^{-1},cb,1,ab\}$ or $\{abc,a^2bc,ac,ab,1\}$, respectively. This is impossible because $c^\b\in \sg{c}$. Thus, $y=1$ and $S^\b=
\{1,a^{-1},b,c^{-1},abc^{-1}\}$. This implies $c^{\b}=c^{-1}$ because $c^\b\in \sg{c}$. Since all involutions of $H$ generate the dihedral subgroup $\sg{a,b}$, $\sg{a,b}$ is characteristic in $H$, and since  $\sg{a,b}$ is dihedral, $\sg{a}$ is characteristic in $H$. Thus, $a^\b\in\sg{a}$ and $b^\b\in\sg{a,b}$, and since $S^\b=
\{1,a^{-1},b,c^{-1},abc^{-1}\}$, we have $a^\b=a^{-1}$, $b^\b=b$ and $(abc)^\b=abc^{-1}$. However, $abc^{-1}=(abc)^\b=a^\b b^\b c^\b=a^{-1}bc^{-1}$, that is, $a^2=a$, contrary the hypothesis $n\geq 3$. This completes the proof.
\end{proof}

\begin{lem}\label{L-1}
Let $p$ be an odd prime, and let
$$H=Q_8\times\Z_p=\sg{a,b,c~|~a^4=b^4=c^p=[a,c]=[b,c]=1, a^2=b^2,a^b=a^{-1}},$$
and $S=\{1,a,c,abc^{-1},bc\}$. Then $\Aut(\H(H,S))=R(H)$ and $H\not\in\BC$.
\end{lem}

\begin{proof} Let $\G=\H(H,S)$ and let $A=\Aut(\G)$.
The lemma holds for $p=3$ and $5$ by {\sc Magma}~\cite{BCP}, and we assume that $p\geq 7$ in the rest of the proof.
Since $A$ is transitive or has the two orbits $H_0$ and $H_1$ as same as $R(H)$, we have $|A_{h_0}|=|A_{k_1}|$ for any $h,k\in H$.

For each $h\in H$, the neighborhood of the vertices $h_0$ and $h_1$ in $\G$ are $\{(sh)_1~|~s\in S\}$ and $\{(s^{-1}h)_0~|~s\in S\}$, respectively. From this, it is easy to list the vertices in $\G$ having distance at most
$2$ from $1_0$ or $c_1$ in Table~\ref{table=1}.

\begin{table}[ht]

\begin{center}
\begin{tabular}{|l|l|l|}

\hline
 $v$    & neighbors of $v$          & vertices having distance 2 from $v$   \\
\hline
\multirow{5}{*}{$1_0$}  & $1_1$ &$(a^{-1})_0$, $(c^{-1})_0$, {\bf $(a^{-1}bc)_0$}, $(b^{-1}c^{-1})_0$\\
\cline{2-3}
& $a_1$ & $a_0$, $(ac^{-1})_0$, $(b^{-1}c)_0$, $(abc^{-1})_0$ \\
\cline{2-3}
 &$c_1$ & $c_0$, $(a^{-1}c)_0$, $(a^{-1}bc^{2})_0$, $(b^{-1})_0$\\
 \cline{2-3}
 &$(abc^{-1})_1$ &$(abc^{-1})_0$, $(bc^{-1})_0$, $(abc^{-2})_0$, $(a^{-1}c^{-2})_0$\\
  \cline{2-3}
 &$(bc)_1$ & $(bc)_0$, $(a^{-1}bc)_0$, $b_0$,  $(ac^{2})_0$\\
\hline
\multirow{5}{*}{$c_1$}  & $c_0$ &$(ac)_1$, $(c^2)_1$, $(ab)_1$, $(bc^2)_1$\\
\cline{2-3}
& $(a^{-1}c)_0$ & $(a^{-1}c)_1$, $(a^{-1}c^2)_1$, $(b^{-1})_1$, $(abc^{2})_1$ \\
\cline{2-3}
 &$1_0$ & $1_1$, $a_1$, $(abc^{-1})_1$, $(bc)_1$\\
 \cline{2-3}
 &$(a^{-1}bc^2)_0$ &$(a^{-1}bc^2)_1$, $(bc^2)_1$, $(a^{-1}bc^3)_1$, $(a^{-1}c^{3})_1$\\
  \cline{2-3}
 &$(b^{-1})_0$ & $(b^{-1})_1$, $(a^{-1}b)_1$, $(b^{-1}c)_1$,  $(ac^{-1})_1$\\
\hline

\end{tabular}
\end{center}
\vskip -0.5cm
\caption{The vertices in $\G$ having distance at most $2$ from $1_0$ or $c_1$.}\label{table=1}
\end{table}

Furthermore, we have the following equations:
\begin{eqnarray}
\G((b^{-1})_1)=\{(b^{-1})_0,(ab)_0,(b^{-1}c^{-1})_0,(a^{-1}c)_0,(b^2c^{-1})_0\},\label{neighbours1}\\
\G((bc^2)_1)=\{(bc^2)_0,(a^{-1}bc^2)_0,(bc)_0,(ac^3)_0,c_0\}.\label{neighbours2}
\end{eqnarray}

By Table~\ref{table=1}, there are exactly two $4$-cycles $C_1$ and $C_2$ passing through $1_0$:
$$C_1=(1_0,1_1,(a^{-1}bc)_0,(bc)_1),\ \ \ \ C_2=(1_0,a_1,(abc^{-1})_0,(abc^{-1})_1),$$
and there are exactly two $4$-cycles $C_3$ and $C_4$ passing through $c_1$:
$$C_3=(c_1,c_0,(bc^2)_1,(a^{-1}bc^2)_0),\ \ \ \ C_4=(c_1,(b^{-1})_0,(b^{-1})_1,(a^{-1}c)_0).$$
We depicted, using Table~\ref{table=1}, Eqs.~(\ref{neighbours1}) and~\eqref{neighbours2},
an induced subgraph of $\G$ in Figure~\ref{Fig-2}.
\begin{figure}[h]
\begin{center}
\unitlength 1mm
\begin{picture}(100,50)

\put(50, 0) {\circle*{1}} \put(49,-5){$1_0$}
\put(10, 15){\circle*{1}} \put(8,10){$(bc)_1$}
\put(50,0){\line(-8, 3){40}}

\put(25, 15){\circle*{2}} \put(24,11){$1_1$}
\put(50,0){\line(-5, 3){25}}

\put(40, 15){\circle*{1}} \put(36,13){$a_1$}
\put(50,0){\line(-2, 3){10}}

\put(55, 15){\circle*{1}} \put(55,12){$(abc^{-1})_1$}
\put(50,0){\line(1, 3){5}}

\put(80, 15){\circle*{1}} \put(81,11){$c_1$}
\put(50,0){\line(2, 1){30}}

\put(5,30){\circle*{1}} \put(0,32){{\small$(bc)_0$}}
\put(10,15){\line(-1, 3){5}}
\put(15,30){\circle*{1}} \put(9,32){{\small$(a^{-1}bc)_0$}}
\put(10,15){\line(1, 3){5}} \put(25,15){\line(-2, 3){10}}
\put(35,30){\circle*{2}} \put(25,32){{\small$(b^{-1}c^{-1})_0$}}
\put(25,15){\line(2, 3){10}}
\put(50,30){\circle*{1}} \put(44,32){{\small $(abc^{-1})_0$}}
\put(40,15){\line(2, 3){10}} \put(55,15){\line(-1, 3){5}}

\put(60,30){\circle*{1}} \put(53,25){{\small$(b^{-1})_0$}}
\put(80,15){\line(-4, 3){20}}
\put(80,30){\circle*{1}} \put(75,32){{\small $(a^{-1}c)_0$}}
\put(80,15){\line(0, 1){15}}
\put(70,30){\circle*{1}} \put(67,32){{\small $c_0$}}
\put(80,15){\line(-2, 3){10}}
\put(90,30){\circle*{1}} \put(91,32){{\small $(a^{-1}bc^2)_0$}}
\put(80,15){\line(2, 3){10}}
\put(42,44){\circle*{1}} \put(40,47){{\small $(b^{-1}c^{-1})_1$}}
\put(35,30){\line(1,2){7}} \put(50,30){\line(-4,7){8}}

\put(70,44){\circle*{2}} \put(65,47){{\small $(b^{-1})_1$}}
\put(60,30){\line(5,7){10}} \put(80,30){\line(-5,7){10}}
\put(35,30){\line(5, 2){35}}

\put(80,44){\circle*{1}} \put(80,47){{\small $(bc^{2})_1$}}
\put(70,30){\line(5,7){10}} \put(90,30){\line(-5,7){10}}
\put(5,30){\line(75, 14){75}}

\end{picture}
\end{center}\vspace{-0cm}
\caption{An induced subgraph of $\G$.}\label{Fig-2}
\end{figure}
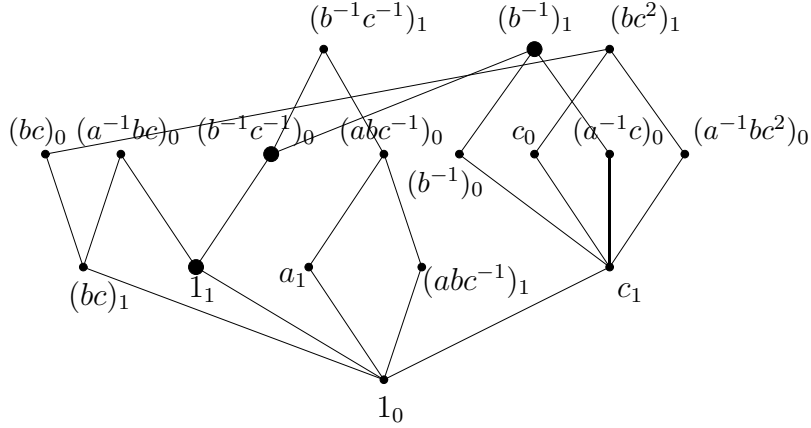

There exist $4$-cycles passing through $1_0$ and any given vertex in $\G(1_0)$ except $c_1$. Then $A_{1_0}\leq A_{c_1}$ and hence $A_{1_0}$ fixes $\{C_1,C_2\}$ and $\{C_3,C_4\}$ setwise. Furthermore, $A_{1_0}$ fixes $\{(b^{-1})_1,(bc^2)_1\}$ setwise because these two vertices are antipodal to $c_1$ in $C_3$ and $C_4$ respectively, and since $|A_{h_0}|=|A_{k_1}|$ for any $h,k\in H$, we have $A_{1_0}=A_{c_1}$.

We first prove that $A_{1_0}$ fixes the $4$-cycle $C_1$ setwise. Recall that $A_{1_0}$ fixes $\{C_1,C_2\}$ setwise. Suppose to the contrary that $\a\in A_{1_0}$ interchanges $C_1$ and $C_2$.  Then $\{1_1,(bc)_1\}^\a=\{a_1,(abc^{-1})_1\}$, and since $A_{1_0}$ fixes $\{(b^{-1})_1,(bc^2)_1\}$ setwise, we have $\{[\G(1_1)\cup \G((bc)_1)]\cap [\G((b^{-1})_1)\cup \G((bc^2)_1)]\}^\a=[\G(a_1)\cup \G((abc^{-1})_1)]\cap [\G((b^{-1})_1)\cup \G((bc^2)_1)]$, which is impossible because $[\G(a_1)\cup \G((abc^{-1})_1)]\cap [\G((b^{-1})_1)\cup \G((bc^2)_1)]=\emptyset$ and  $(bc)_0\in \G((bc)_1)\cap \G((bc^2)_1)$ by Table~\ref{table=1}, Eqs.~(\ref{neighbours1}) and~\eqref{neighbours2}. Thus, $A_{1_0}$ fixes $C_1$ setwise.

Now we prove that $A_{1_0}$ fixe $C_1$ pointwise. Since $A_{1_0}$ fixes $C_1$ setwise, it fixes $C_2$ setwise, implying $A_{1_0}$ fixes $(abc^{-1})_0$. Suppose to the contrary that $\b\in A_{1_0}$ interchanges $1_1$ and $(bc)_1$. By Table~\ref{table=1}, $\G((bc)_1)\cap [\G((b^{-1})_1)\cup \G((bc^2)_1)]=\{(bc)_0\}$ and  $\G(1_1)\cap [\G((b^{-1})_1)\cup \G((bc^2)_1)]=\{(b^{-1}c^{-1})_0\}$. Since $A_{1_0}$ fixes $\{(b^{-1})_1,(bc^2)_1\}$ setwise, $\b$ interchanges $(bc)_0$ and $(b^{-1}c^{-1})_0$, implying $\G((bc)_0)^\b=\G((b^{-1}c^{-1})_0)$. Since $A_{1_0}$ fixes $(abc^{-1})_0$, we have $[\G((abc^{-1})_0)\cap\G((bc)_0)]^\b=\G((abc^{-1})_0)\cap\G((b^{-1}c^{-1})_0)$. It is easy to see that $(b^{-1}c^{-1})_1\in \G((abc^{-1})_0)\cap\G((b^{-1}c^{-1})_0)$. Then $\G((abc^{-1})_0)\cap\G((bc)_0)\not=\emptyset$, and there is $s,t\in S$ such that $sabc^{-1}=tbc$, that is, $s^{-1}t=ac^{-2}\in S^{-1}S$, which is impossible because $S^{-1}S=\{1,a,c,abc^{-1},bc,a^{-1},a^{-1}c,bc^{-1},a^{-1}bc,c^{-1},ac^{-1},abc^{-2},b,b^{-1}c,a^{-1}bc^2$,
$ac^2$, $b^{-1}c^{-1}$, $b^{-1}$, $a^{-1}c^{-2}\}$. Thus, $A_{1_0}$ fixes $C_1$ pointwise, and hence $A_{1_0}=A_{1_1}=A_{(bc)_1}=A_{(a^{-1}bc)_0}$.

Since $A_{1_0}=A_{c_1}$, we have $A_{1_0}^{R(c^{-1})}=A_{c_1}^{R(c^{-1})}$, and so $A_{c^{-1}_0}=A_{1_1}=A_{1_0}$. Similarly, we have $A_{1_0}=A_{(b^{-1}c^{-1})_0}$ because $A_{1_0}=A_{(bc)_1}$. It follows that $A_{1_0}=A_{x_0}$ for any $x\in T:=\{c^{-1}, b^{-1}c^{-1},a^{-1}bc\}$, and an easy inductive argument implies that $A_{1_0}=A_{x_0}$ for any $x\in \sg{T}=H$. Thus, $A_{1_0}$ fixes $H_0$ pointwise.
Also, since $A_{1_0}=A_{1_1}$ implies $A_{h_0}=A_{h_1}$ for any $h\in H$, it follows that
$A_{1_0}$ fixes $H_1$ pointwise too.
Thus, $A_{1_0}=1$.

\medskip
To finish the proof, we are left with showing that $A$ is not vertex-transitive. Suppose to the contrary that $A$ is vertex-transitive. Since $A_{1_0}=1$, we have $|A|=|V(\G)|=2|R(H)|$ and hence
$R(H)\unlhd A$. By Proposition~\ref{ZF}, there exists $\d_{\b,x,y}\in A$
for some $\b\in \Aut(H)$ and $x,y\in H$ such that $S^{\b}=y^{-1}S^{-1}x$. By the transitivity of
$R(H)$ on $H_1$, we may assume $1_0^{\d_{\b,x,y}}=1_1$, anb by Eq.~\eqref{Eq-delta}, $1_0^{\d_{\b,x,y}}=(x1^{\b})_1=1_1$, forcing $x=1$. Recall that $S=\{1,a,c,abc^{-1},bc\}$, and $S^{\b}=y^{-1}S^{-1}$, that is,
\begin{equation}\label{Eq-p-gp-1}
S^\b=\{1^{\b},a^{\b},c^{\b},(abc^{-1})^{\b},(bc)^{\b}\}=
y^{-1}\{1,a^{-1},c^{-1},a^{-1}bc,b^{-1}c^{-1}\}.
\end{equation}
Since $1\in S$, we have $1\in S^\b$ and so $y^{-1}=1,a,c, abc^{-1}$ or $bc$.

Note that $H=Q_8\times\Z_p=\sg{a,b,c~|~a^4=b^4=c^p=[a,c]=[b,c]=1, a^2=b^2,a^b=a^{-1}}$. Then $Q_8$ and $\Z_p$ are characteristic in $H$, and hence $c^\b\in\sg{c}$ and $\sg{a,b}^\b=\sg{a,b}$.

Let $y^{-1}=a, abc^{-1}$ or $bc$. By Eq~(\ref{Eq-p-gp-1}), $S^\b=\{a,1,ac^{-1},bc,ab^{-1}c^{-1}\}$, $\{abc^{-1}$, $b^{-1}c^{-1}$, $abc^{-2}$, $1$, $ac^{-2}\}$ or $\{bc,ba^{-1}c,b,a^{-1}c^2,1\}$, which is impossible because $c^\b\in \sg{c}$.

Let $y=1$. Then $S^\b=\{1,a^{-1},c^{-1},a^{-1}bc,b^{-1}c^{-1}\}$ by Eq.~(\ref{Eq-p-gp-1}).
Since $a^{\b}\in\sg{a,b}$ and $c^{\b}\in \sg{c}$, we have $a^{\b}=a^{-1}$, $c^{\b}=c^{-1}$
and $\{(abc^{-1})^{\b},(bc)^{\b}\}=
\{a^{-1}bc,b^{-1}c^{-1}\}$, and since $(bc)^{\b}=b^{\b}c^{-1}\in \sg{a,b}c^{-1}$, we have $(bc)^{\b}=b^{-1}c^{-1}$. Hence $b^{\b}=b^{-1}$ and $(abc^{-1})^{\b}=a^{-1}bc$.
However, $a^{-1}bc=(abc^{-1})^{\b}=a^{\b} b^{\b} (c^{-1})^{\b}=a^{-1} b^{-1} c$, forcing $b^2=1$, a contradiction.

Let $y^{-1}=c$. Then $S^\b=\{c,a^{-1}c,1,a^{-1}bc^2,b^{-1}\}$, yielding that $a^{\b}=b^{-1}$ and $c^{\b}=c$.
Since $(abc^{-1})^{\b}=(ab)^\b c^{-1}\in S^\b$, we have $(ab)^\b c^{-1}=a^{-1}c$ or $a^{-1}bc^2$, forcing $c^2=1$ or $c^3=1$, contradicting $p\geq 7$. This completes the proof.
\end{proof}

To end this section, we describe some Haar graphs of small orders that are not vertex-transitive, and this can be checked easily by the computer software {\sc Magma}~\cite{BCP}.

\begin{lem}\label{smallorders} Let $G=H_i$ and $S$ be given in the following table for each $1\leq i\leq 9$:
\begin{table}[h]
\begin{center}
\begin{tabular}{|l|l|l|}

\hline
$i$ &$H_i$          & $S$   \\
\hline
1&$\sg{a,b,c~|~a^8,b^2,c^2,[a,c],[b,c],a^b=a^{-1}}\cong D_8\times\Z_2$ & $\{1,a,b,c,ab,abc\}$\\
\hline
2& $\sg{a,b,c~|~a^4,b^2,c^2,[a,b],[a,c],[b,c]=a^2}$ & $\{1,a,b,ab,ac,abc\}$\\
\hline
3&$\sg{a,b \mid a^8,b^2=a^4, a^b=a^{-1}}$ & $\{1,a,b,a^5,ab,a^5b\}$\\
\hline
4&$\sg{a,b \mid a^8,b^2,a^b=a^3}$ & $\{1,a,b,ab\}$\\
\hline
5&$\langle a,b,c,d\mid a^4,b^4,c^2,d^2,a^2=b^2,a^b=a^{-1},[a,c],[a,d],$
&$\{1,a,b,b^{-1},ab,ac,$\\
& ~~~~~~~~~~~~~$[b,c],[b,d],[c,d]\rangle\cong Q_8\times\Z_2\times\Z_2$& $bd,abd\}$\\
\hline
6&$\sg{a,b,c \mid a^4,b^4,c^3, a^2=b^2,a^b=a^{-1},a^c=b^{\pm1},b^c=a^{\pm1}b}$ &
$\{1,a,bc,abc\}$\\
\hline
7&$\sg{a,b,c\mid a^2,b^2,c^3,[a,b],a^c=b,b^c=ab}\cong A_4$ &$\{1,a,c,abc\}$\\
\hline
8&$\sg{a,g\mid a^5,g^4,a^g=a^{2}}\cong F_{20}$ &$\{1,a,g\}$\\
\hline
9&$\sg{a,c,b\mid a^p,c^p,b^2,[a,c],a^b=a^{-1},c^b=c^{-1}}\cong \Z_p^2\rtimes\Z_2$, $p=3,5$ &$\{1,a,c,b,ab,cb\}$\\
\hline

\end{tabular}
\end{center}
\vskip -0.5cm
\end{table}

\noindent Then $\Aut(\H(G,S))$ is not vertex-transitive and $G\not\in\BC$.
\end{lem}

\section{Proof of Theorem~\ref{1}}

In this section, we aim to prove Theorem~\ref{1}. First, we need two lemmas.

\begin{lem}\label{lem=2-group}
A non-abelian $2$-group belongs to the class $\BC$ if and only if it is isomorphic to $D_8$, $Q_8$ or $Q_8\times\Z_2$.
\end{lem}

\begin{proof} By Proposition~\ref{Prop=Inner abelian group}, $D_8\in \BC$ and
$Q_8\in\BC$. For $Q_8\times\Z_2$, let $\G=\H(Q_8\times\Z_2,S)$
be a Haar graph with $1\in S$. If
$\G$ is not connected, then $\sg{S}<Q_8\times\Z_2$
(see~\cite[Lemma~1~(i)]{EP}), and either
$\sg{S}$ is abelian or $\sg{S}\cong Q_8$.
This implies that the Haar graph $\H(\sg{S},S)$ is a Cayley graph, and since $\G$ is a union of components with each isomorphic to $\H(\sg{S},S)$, $\G$ is a Cayley graph. If $\G$ is connected, a computation by {\sc Magma}~\cite{BCP} shows that all connected Haar graphs of $Q_8\times\Z_2$
are Cayley graphs. Thus, $Q_8\times\Z_2\in\BC$.

Let $H$ be a non-abelian $2$-group and $H\in \BC$. To prove the necessity, it suffices to show that $H\cong D_8, Q_8$ or $Q_8\times\Z_2$.

\medskip
\noindent {\bf Case 1:} $|H|\leq 8$.

Since $H$ is a non-abelian $2$-group, we have $H\cong D_8$ or $Q_8$.

\medskip
\noindent {\bf Case 2:} $|H|=16$.

Note that all non-abelian groups of order $16$ can be found in~\cite{HS} (this can also be obtained by the computer software {\sc Magma}~\cite{BCP}). By Proposition~\ref{Prop=2}~(iii), $H$ has a subgroup isomorphic to $D_8$ or $Q_8$, and hence $H\cong H_i$ for some $1\leq i\leq 6$:
$$
\begin{array}{ll}
&H_1=\sg{a,b,c~|~a^8=b^2=c^2=[a,c]=[b,c]=1,a^b=a^{-1}}~(\cong D_8\times\Z_2);\\
&H_2=\sg{a,b,c~|~a^4=b^2=c^2=[a,b]=[a,c]=1,[b,c]=a^2};\\
&H_3=\sg{a,b \mid a^8=1,b^2=a^4, a^b=a^{-1}};\\
&H_4=\sg{a,b \mid a^8=b^2=1,a^b=a^3}; \ \ \ H_5=D_{16}; \ \ \ H_6=Q_8\times \Z_2.\\

\end{array}$$
By Proposition~\ref{Prop=1}, $H_5\not\in \BC$, and by Lemma~\ref{smallorders}, $H_i\not\in\BC$ for each $1\leq i\leq 4$. It follows that $H\cong H_6=Q_8\times\Z_2$.

\medskip
\noindent{\bf Case 3:} $|H|\geq 32$.

Since $H\in \BC$, Proposition~\ref{L1} implies that each subgroup of $H$ belongs to $\BC$. If each subgroup of $H$ of order $32$ is abelian, then $H$ has an inner abelian subgroup of order at least $64$, which is impossible by Proposition~\ref{Prop=Inner abelian group}. Thus, $H$ has a non-abelian subgroup of order $32$, say $L$.
Then $L\in \BC$. Similarly, $L$ has a non-abelian subgroup of order $16$, and by the proof of Case~2, each non-abelian subgroup of $L$ of order $16$ is isomorphic to $Q_8\times\Z_2$. By checking the non-abelian groups of order $32$ listed in~\cite{HS}, we have $L\cong Q_8\times\Z_2\times\Z_2$, and by Lemma~\ref{smallorders}, $L\notin\BC$, a contradiction.
\end{proof}

\begin{lem}\label{lem=2,p-group}
Let $p$ be  an odd prime, and let $H$ be a non-abelian $\{2,p\}$-group with $p\mid |H|$.
Then $H\in\mathcal{BC}$ if and only if $H\cong D_6$ or $D_{10}$.
\end{lem}

\begin{proof} By~Proposition~\ref{Prop=1}, $D_{6}\in \BC$
and $D_{10}\in \BC$. Let $H$ be a non-abelian $\{2,p\}$-group with $p\mid |H|$. To prove the necessity, suppose to the contrary that $H$ is a minimal counterexample, that is, $H\in\BC$ has the smallest order such that $H\not\cong D_6$ or $D_{10}$.

Denote by $P$ and $P_2$ a Sylow $p$-subgroup and a Sylow $2$-subgroup
of $H$, respectively. Then $H=PP_2$, and by Proposition~\ref{L1}, $P\in \BC$ and $P_2\in \BC$.
By Proposition~\ref{Prop=2}~(ii),
$P$ is abelian, and by Lemma~\ref{lem=2-group}, either $P_2$ is abelian or $P_2\cong D_8$, $Q_8$, or $Q_8\times\Z_2$.
Now we consider the two cases depending whether $P_2$ is normal in $H$.
\medskip

\noindent {\bf Case~1:} $P_2\unlhd H$.

Suppose that $P_2$ is abelian. It follows from Proposition~\ref{Prop=2}~(iii)
that $H$ has a subgroup $D_{2p}$ with $p=3$ or $5$.
Since $P_2$ is the unique Sylow $2$-subgroup of $H$, all
involutions of $D_{2p}$ are contained in $P_2$,
and since $D_{2p}$ can be generated by its two involutions,
we have $D_{2p}\leq P_2$, which is impossible.
Hence $P_2$ is non-abelain. By Proposition~\ref{L1}, $P_2\in\BC$, and by Lemma~\ref{lem=2-group},
$P_2\cong D_8$, $Q_8$ or $Q_8\times\Z_2$.
It follows $H=P_2\rtimes P$, and by the minimality of $H$, $|P|=p$. Thus, $H=P_2\rtimes P\cong D_8\rtimes\Z_p$,
$Q_8\rtimes\Z_p$ or $(Q_8\times\Z_2)\rtimes\Z_p$.

Consider the centralizer $C_H(P_2)$ of $P_2$ in $H$.
If $P\leq C_H(P_2)$, then $H=P_2\times P\cong D_8\times\Z_p$,
$Q_8\times \Z_p$, or $Q_8\times\Z_2\times\Z_p$,
and by Lemmas~\ref{L-2} and \ref{L-1}, $D_8\times\Z_p\not\in\BC$ and  $Q_8\times\Z_p\not\in\BC$. Also we have
 $Q_8\times\Z_2\times\Z_p\not\in\BC$ because otherwise $Q_8\times\Z_2\times\Z_p\in\BC$ implies  $Q_8\times\Z_p\in\BC$ by Proposition~\ref{L1}. Thus, $H\not\in \BC$, a contradiction. Hence $P\nleq C_H(P_2)$,
and since $|P|=p$, we have $P\cap C_H(P_2)=1$ and so $C_H(P_2)\leq P_2$.
Note that $\Aut(D_8)\cong D_8$, $\Aut(Q_8)\cong S_4$, and $\Aut(Q_8\times\Z_2)\cong\Z_2^3\rtimes S_4$. Since $N_H(P_2)/C_H(P_2)=H/C_H(P_2)\lesssim \Aut(P_2)$, we have $P_2\cong Q_8$ or $Q_8\times\Z_2$, and $P\cong \Z_3$, which implies that a generator of $P$ induces (by conjugacy) an automorphism of $P_2$ of order $3$.
For $P_2\cong Q_8$, let $P_2=\sg{a,b \mid a^4=b^4=1, a^2=b^2, a^b=a^{-1}}$, and let $\a$ be the automorphism of $P_2$ of order $3$ induced by the map $a\mapsto b$ and $b\mapsto ab$. Since all the automorphisms of $P_2$ of order $3$ are conjugate in $\Aut(P_2)$, we have $H\cong P_2\rtimes\sg{\a}=\sg{a,b,\a \mid a^4=b^4=\a^3=1, a^2=b^2,a^b=a^{-1},a^{\a}=b,b^{\a}=ab}$.
By Lemma~\ref{smallorders}, $H\cong H_6$ and $H\notin\BC$,
a contradiction. Similarly, for $P_2\cong Q_8\times\Z_2$, let $P_2=\sg{a,b \mid a^4=b^4=1, a^2=b^2, a^b=a^{-1}} \times\sg{c}$ with $\sg{c}\cong\Z_2$, and let $\a$ be the automorphism of $P_2$ of order $3$ induced by the map $a\mapsto b$, $b\mapsto ab$ and $c\mapsto c$.
 It follows that $H\cong (\sg{a,b}\times \sg{c})\rtimes\sg{\a}\cong H_6\times\Z_2$ and hence $H\notin \BC$ as $H_6\not\in\BC$, a contradiction.
\medskip

\noindent {\bf Case~2:} $P_2\ntrianglelefteq H$.

Since $P$ is abelian, we have $P\leq C_H(P)\leq N_H(P)$. If $C_H(P)=N_H(P)$,
then the Burnside's $p$-nilpotency criterion implies that
$P$ has a normal complement $N$, that is,
$N$ is a normal Sylow $2$-subgroup of $H$, which is impossible by Case~1.
Thus, we may assume that $P\leq C_H(P)<N_H(P)$, and hence there exists a $2$-element $g$ such that $g\in N_H(P)$ and $g\not\in C_H(P)$. In particular, $P\rtimes\sg{g}$ is a non-abelian subgroup of $H$ and $p \ |\ |P\rtimes\sg{g}|$. By the minimality of $H$,
either $H=P\rtimes\sg{g}$, or $H>P\rtimes\sg{g}\cong D_6$ or $D_{10}$.

\medskip
\noindent {\bf Subcase 2.1:} $H=P\rtimes\sg{g}$.

In this case, $P_2=\langle g\rangle$ and $P\unlhd H$. By Proposition~\ref{Prop=Inner abelian group}, $H$ contains a proper subgroup $K$ such that $K\cong D_6$ or $D_{10}$. Then $p=3$ or $5$. Let $K=\langle a,b\ |\ a^p=b^2=1,a^b=a^{-1}\rangle$. We may assume $b\in P_2$, and since $P_2=\langle g\rangle$, $b$ is the unique involution in $P_2$. Furthermore, $PK\leq H$ is non-abelian. Let $|P|=p^s$ for some $s\geq 1$. Then $|PK|=2p^s$.

If $s=1$ then $P\cong \Z_p$ and hence $P\leq K$. Since $K\cong D_6$ or $D_{10}$, we have $K<H$, and since $K/P<H/P\cong P_2$, $H$ contains a subgroup of order $4p$ with $K$ as a subgroup of index $2$. By  the minimality of $H$, we have $P_2=\langle g\rangle\cong \Z_4$. It follows $b=g^2$, and since $a^b=a^{-1}$, we have $p=5$ and $a^g=a^2$ or $a^3$. This implies that $H\cong F_{20}$ and by Lemma~\ref{smallorders}, $H\not\in\BC$, a contradiction. Thus, $s\geq 2$.

Suppose $s\geq 3$. Note that $P\unlhd H$ and $K=\langle a,b\rangle\cong D_{2p}<H$ with $o(a)=p$. By the minimality of $H$, we have $H=PK$. Then $|H|=2p^s$ and $P_2=\langle g\rangle=\sg{b}\cong\Z_2$. If $P$ is cyclic, it is easy to see that $H$ is dihedral, which is impossible by Proposition~\ref{Prop=1}. Thus, $P$ is not cyclic,
and since $P$ is abelian,
Proposition~\ref{prop=abelian p-group} implies that
$P$ has an element $c$ of order $p$ with
$\langle c\rangle\cap \langle a\rangle=1$.
If $c^b\not\in \langle c\rangle$ then $\langle c,c^b,b\rangle$ is a non-abelian subgroup of order $2p^2$, and by the minimality of $H$, we have $H=\langle c,c^b,b\rangle$, which is impossible because $|H|=2p^s$ with $s\geq 3$. Similarly, if  $c^b\in \langle c\rangle$ then $\langle a,c,b\rangle$ is a non-abelian subgroup of order $2p^2$, which is also impossible.

Thus, $s=2$ and $|H|=2p^2$. From the elementary group theory we know that up to isomorphism there are three non-abelian groups of order $2p^2$ defined as:
$$
\begin{array}{lll}
H_{1}(p)&=&\langle a,b\mid a^{p^2}=b^{2}=1, b^{-1}ab=a^{-1}\rangle;\\
H_{2}(p)&=&\langle a,b,c\mid a^{p}=b^{p}=c^2=[a,b]=1, c^{-1}ac=a^{-1}, c^{-1}bc=
b^{-1}\rangle;\\
H_{3}(p)&=&\langle a,b,c \mid a^{p}=b^{p}=c^2=1,[a,b]=[a,c]=1,c^{-1}bc=
b^{-1}\rangle.
\end{array}
$$

Thus, $H\cong H_1(p), H_2(p)$ or $H_3(p)$. Note that $H_{1}(p)\cong D_{2p^2}$ and $H_{3}(p)\cong D_{2p}\times \Z_p$. By Proposition~\ref{Prop=1} and Lemma~\ref{L-2}, $H_{1}(p)\not\in\BC$ and $H_{3}(p)\not\in\BC$. Recall that $p=3$ or $5$. By Lemma~\ref{smallorders},  $H_2(p)\not\in \BC$. It follows that $H\not\in \BC$, a contradiction.

\medskip
\noindent {\bf Subcase 2.2:} $H>P\rtimes\sg{g}\cong D_6$ or $D_{10}$.

Clearly, $\sg{g}\cong\Z_2$, $p=3$ or $5$, and $P\cong \Z_3$ or $\Z_5$. Let $P=\langle a\rangle\cong\Z_p$. Without any loss of generality, we may assume that $g\in P_2$. First we prove two claims.

\medskip
\noindent {\bf Claim 1:} $P_2\cong D_8$, $Q_8$, or $Q_8\times\Z_2$.

Recall that either $P_2$ is abelian, or $P_2\cong D_8, Q_8$ or $Q_8\times\Z_2$.

Suppose that $P_2$ is abelian. Then $P_2\leq C_H(P_2)\leq N_H(P_2)$. Since$|H:P_2|=p$ and $P_2\ntrianglelefteq H$, we have $P_2=C_H(P_2)=N_H(P_2)$. By the Burnside's $p$-nilpotency criterion, $P$ is the normal complement of $P_2$ in $H$, that is,
$P\unlhd H$ and $H=P\rtimes P_2$. Note that $\Z_2\cong\sg{g}\cong (P\rtimes\sg{g})/P<H/P\cong P_2$. Then $H/P$ contains a subgroup of order $4$, and hence $H$ has a non-abelian subgroup $L$ with $P\rtimes\sg{g}$ as a subgroup of index $2$. By the minimality of $H$, $H=L$ and so $|P_2|=4$. In particular, $P_2\cong\Z_4$ or $\Z_2^2$.

Let $P_2=\langle b\rangle\cong\Z_4$. Then $H=P\rtimes P_2=\sg{a,b,g\mid a^p=b^4=1,g=b^2, a^b=a^i,a^g=a^{-1}}$
with $i\in\Z_p^*$. Since $a^{-1}=a^g=a^{b^2}=a^{i^2}$,
we have $i^2=-1$ in $\Z_p$, and since $p=3$ or $5$, we have
$i=\pm2$ and $p=5$. Hence $H$ is isomorphic to the Frobenius group $F_{20}$ of order 20 and so $H\notin\BC$ by Lemma~\ref{smallorders}, a contradiction.

Let $P_2=\langle b,g\rangle\cong\Z_2^2$. Considering the action of $P_2$ on $P$, we have
$H=P\rtimes P_2\cong\sg{a,b,g\mid a^p=b^2=g^2=[b,g]=1, a^b=a,a^g=a^{-1}}\cong D_{4p}$, and by Proposition~\ref{Prop=1}, $H\notin\BC$, a contradiction.

It follows that $P_2\cong D_8$, $Q_8$, or $Q_8\times\Z_2$, as claimed.

\medskip
\noindent {\bf Claim 2:} $P\unlhd H$.

Let $N$ be a minimal normal subgroup of $H$. Since $P\cong\Z_p$, we have $N=P$ or $N=\Z_2^\ell$ for some $\ell\geq 1$. If $N=P$ then $P\unlhd H$, as claimed. If $N=\Z_2^\ell$, then $(P\rtimes\sg{g})N>P\rtimes\sg{g}$ because $P\rtimes\sg{g}$ ($\cong D_6$ or $D_{10}$) has non-normal Sylow $2$-subgroups. By the minimality of $H$, we have $H=N(P\rtimes\sg{g})$. Since $P_2\ntrianglelefteq H$, we have $N<P_2$ and hence $NP<H$. Clearly, $NP\not\cong D_6$ or $D_{10}$, and the minimality of $H$ implies that $NP$ is abelian. It follows $NP=N\times P$ and $H=(N\times P)\rtimes\sg{g}$. Thus, $P\unlhd H$, as claimed.

\medskip
By Claim~2, $PK\leq H$ for any subgroup $K\leq H$, and by Claim~1, $P_2\cong D_8$, $Q_8$, or $Q_8\times\Z_2$. If $P_2\cong Q_8\times\Z_2$ then $H$ has a proper subgroup isomorphic to $P\rtimes Q_8$, and the minimality of $H$ implies that either $P\rtimes Q_8$ is abelian, or $P\rtimes Q_8\cong D_6$ or $D_{10}$, both of which are impossible. If $P_2\cong Q_8$, then $P\rtimes\langle a\rangle$ is a proper subgroup of $H$ for any element $a$ of order $4$ in $P_2$, and by the minimality of $H$,  $P\rtimes\langle a\rangle$ is abelian because $P\rtimes\langle a\rangle\ncong D_6$ or $D_{10}$. This implies that $[P,P_2]=1$ as $P_2\cong Q_8$, and hence $H=P\times P_2$, which is impossible because $P\rtimes\langle g\rangle\cong D_6$ or $D_{10}$. If $P_2\cong D_8$, let $P_2=\sg{a,g\mid a^4=g^2=1,a^g=a^{-1}}$ and $P=\sg{b}$, and a similar argument
as above implies $[P,\sg{a}]=1$. Since $P\rtimes\sg{g}\cong D_6$ or $D_{10}$, we have $b^g=b^{-1}$. It follows that
$$H\cong\Z_p\rtimes D_8=\sg{a,b,g\mid b^p=a^4=g^2=1,a^g=a^{-1},b^a=b,b^g=b^{-1}}\cong D_{8p}.$$
By Proposition~\ref{Prop=1}, $H\notin\BC$,  a contradiction.

\end{proof}

Now we are ready to prove Theorem~\ref{1}.
\medskip

\begin{proof}[\bf Proof of Theorem~\ref{1}:] By Lemmas~\ref{lem=2-group} and \ref{lem=2,p-group}, $D_8$, $Q_8$, $Q_8\times\Z_2$, $D_6$ and $D_{10}$ belong to $\BC$. To prove the necessity, let $H$ be a non-abelian group with $H\in \BC$.

By Proposition~\ref{Prop=2}~(i), $H$ is solvable, and by Proposition~\ref{Prop=2}~(iii), $2\mid |H|$.
Let $L$ and $K$ be a Sylow $2$-subgroup and a Hall $2'$-subgroup of $H$, respectively. By Proposition~\ref{Prop=2}, $K\in\BC$ and $L\in\BC$. Furthermore, $H=KL$, and by Proposition~\ref{Prop=2}~(iii), $K$ is abelian. Let $p_1,\ldots, p_k$ be all distinct odd prime divisors of $|H|$, and let $P_i$ be a Sylow $p_i$-subgroup of
$H$ contained in $K$ for $1\leq i\leq k$. Then $K=P_1\times P_2\times\cdots\times P_k$.
If $k=0$ then $H=L$ is a $2$-group, and by Lemma~\ref{lem=2-group},
$H\cong D_8$, $Q_8$ or $Q_8\times\Z_2$. If $k=1$ then $H$ is a $\{2,p_1\}$-group, and by Lemma~\ref{lem=2,p-group}, then $H\cong D_6$ or $D_{10}$. In what follows, we assume $k\geq 2$.

If each Hall $\{2,p_i\}$-subgroup of $H$ is abelian for each $1\leq i\leq k$, then
$L$ is abelian and $H=K\times L$, forcing that
$H$ is abelian, a contradiction.
Hence $H$ has a non-abelian Hall $\{2,p_\ell\}$-subgroup
for some prime $p_\ell$, say $M$.
It follows from Proposition~\ref{Prop=2} that
$M\in\BC$ and from Lemma~\ref{lem=2,p-group}
that $M\cong D_6$ or $D_{10}$, yielding that $L\cong\Z_2$.
Hence $K\unlhd H$ and $H=K\rtimes P_2=(P_1\times\cdots\times P_k)\rtimes\Z_2$. It follows that $P_i\unlhd H$ for each $1\leq i\leq k$, and hence $P_iL\leq H$. Furthermore, we may assume $M=P_{\ell}L$. Again by Lemma~\ref{lem=2,p-group}, for each $1\leq i\leq k$ we have either $P_iL=P_i\times L$ (abelian), or $P_iL\cong D_6$ or $D_{10}$.

Suppose $P_jL=P_j\times L$ for some $1\leq j\leq k$. Recall that $M=LP_\ell$ is a Hall $\{2,p_\ell\}$-subgroup of $H$, and $M\cong D_6$ or $D_{10}$. Clearly, $p_\ell\not=p_j$, and $MP_j=LP_\ell P_j=M \times P_j$. Then $H$ contains a subgroup isomorphic to $D_6\times \Z_{p_j}$ or $D_{10}\times \Z_{p_j}$, which is impossible by Lemma~\ref{L-2}. Note that if $p_i\not=3,5$, then $P_iL=P_i\times L$ because $P_iL\ncong D_6$ or $D_{10}$. This implies that $k=2$ as $k\geq 2$, and $\{p_1,p_2\}=\{3,5\}$. Furthermore, $\{P_1L,P_2L\}=\{D_6,D_{10}\}$ and hence  $H=P_1P_2L\cong D_{30}$ because a group of order $15$ must be cyclic, which is impossible by Proposition~\ref{Prop=1}. This completes the proof.
\end{proof}

\end{document}